\documentclass{article}

\usepackage[latin9]{inputenc}
\usepackage{color}
\usepackage{amsmath,amstext,amsthm}
\usepackage[colorlinks=true,citecolor=blue]{hyperref}
\usepackage[square,numbers,sort,compress]{natbib}
\usepackage{amssymb}
\usepackage{graphicx}
\usepackage{comment,sidecap}
\usepackage{esint}
\PassOptionsToPackage{normalem}{ulem}
\usepackage{ulem}
\usepackage[title]{appendix}
\usepackage{framed}
\usepackage{bookmark}

\makeatletter
  \theoremstyle{plain}
  \newtheorem{lem}{\protect\lemmaname}
\theoremstyle{plain}
\newtheorem{thm}{\protect\theoremname}
  \theoremstyle{remark}
  \newtheorem{rem}{\protect\remarkname}

\@ifundefined{date}{}{\date{}}
\newtheorem{example}{Example}

\newcommand{\id}{\mbox{d}}

\newcommand{\St}{\mbox{St}}
\newcommand{\Rey}{\mbox{Re}}
\newcommand{\e}{\mbox{e}}

\makeatother

\providecommand{\lemmaname}{Lemma}
\providecommand{\remarkname}{Remark}
\providecommand{\theoremname}{Theorem}

\begin{document}
\title{The Maxey--Riley Equation: Existence, Uniqueness and Regularity of
Solutions}

\author{Mohammad Farazmand$^{1,2}$\footnote{Corresponding author; email address: farazmam@ethz.ch}, George Haller$^{2}$\\
 {\small $^{1}$Department of Mathematics, ETH Zurich, Rämistrasse
101, 8092 Zurich Switzerland} \\
 {\small $^{2}$Institute of Mechanical Systems, ETH Zurich, Tannenstrasse 3, 8092 Zurich, Switzerland}}
\maketitle
\begin{abstract}
The Maxey--Riley equation describes the motion of an inertial (i.e.,
finite-size) spherical particle in an ambient fluid flow. The equation
is a second-order, implicit integro-differential equation with a singular
kernel, and with a forcing term that blows up at the initial time.
Despite the widespread use of the equation in applications, the basic
properties of its solutions have remained unexplored. Here we fill
this gap by proving local existence and uniqueness of mild solutions.
For certain initial velocities between the particle and the fluid,
the results extend to strong solutions. We also prove continuous differentiability
of the mild and strong solutions with respect to their initial conditions.
This justifies the search for coherent structures in inertial flows
using the Cauchy--Green strain tensor.
\end{abstract}

\section{Introduction}

The Maxey--Riley equation \cite{Tchen,IP_corrsin,MR} describes the motion
of a small but finite-sized rigid sphere through a fluid. The equation
is widely used to study the motion of a finite-size (or inertial)
particle immersed in a non-uniform fluid. The behavior of such particles
is of interest in various environmental and engineering problems,
e.g., clustering of garbage patches in the oceans \cite{garbagePatch}
and dispersion of airborne pollutants \cite{airPollution}.

A first attempt to derive the equation of motion of an inertial particle in a non-uniform flow appears in \cite{Tchen}. \citet{Tchen} wrote the Basset--Boussinesq-Oseen equation (governing the motion of a small spherical particle in a quiescent fluid \cite{boussinesq,basset2,oseen1927}) in a frame co-moving with a fluid parcel in an unsteady flow, accounting for various forces that arise in such non-inertial frames. Later, the exact form of the forces exerted on the particle were debated and corrected by several authors \cite[see, e.g.,][]{IP_corrsin,MR,auton1988}. Today, the most widely accepted form of the equations is the Maxey--Riley (MR) equation \cite{MR} with the corrections due to \citet{auton1988} and \citet{MR_initCond}.

To recall the exact form of the MR equation, we let $u:\mathcal{D}\times\mathbb{R}^{+}\rightarrow\mathbb{R}^{n}$
denote the velocity field describing the flow of a fluid in an open
spatial domain $\mathcal{D}\subseteq\mathbb{R}^{n}$, where $n=2$
or $n=3$ for two- or three-dimensional flows, respectively. A fluid
trajectory is then the solution of the differential equation $\dot{x}=u(x,t)$
with some initial condition $x(t_0)=x_{0}$. A spherical inertial particle,
however, follows a different trajectory $y(t)\in\mathcal{D}$, which
satisfies the MR equation 
\begin{align}
\nonumber \\
\ddot{y}= & \,\frac{R}{2}\frac{D}{Dt}\left(3u(y,t)+\frac{\gamma}{10}\mu^{-1}\Delta u(y,t)\right)+\left(1-\frac{3R}{2}\right)g\nonumber \\
\  & -\mu\left(\dot{y}-u(y,t)-\frac{\gamma}{6}\mu^{-1}\Delta u(y,t)\right)\nonumber \\
 & -\kappa\mu^{1/2}\left\{\int_{t_0}^{t}\frac{\dot w(s)}{\sqrt{t-s}}\id s+\frac{w(t_0)}{\sqrt{t-t_0}} \right\},\label{eq:MR_dimless0}
\end{align}
where 
\begin{equation}
w(t)=\dot{y}(t)-u(y(t),t)-\frac{\gamma}{6}\mu^{-1}\Delta u(y(t),t).
\label{eq:w}
\end{equation}
The initial conditions for the inertial particle are given as $y(t_0)=y_{0}$
and $\dot y(t_0)=v_{0}$. The material derivative $\frac{D}{Dt}\doteq\partial_{t}+u\cdot\nabla$
denotes the time derivative along a fluid trajectory.

All the variables and parameters in equations \eqref{eq:MR_dimless0} and \eqref{eq:w} are dimensionless, nondimensionalized by characteristic length scale $L$, characteristic velocity $U$ and characteristic time scale $T=L/U$ of the fluid flow. The dimensionless parameters are 
\begin{equation}
R=\frac{2\rho_{f}}{\rho_{f}+2\rho_{p}},\ \mu=\frac{R}{St},\ \kappa=\sqrt{\frac{9R}{2\pi}},\ \gamma=\frac{9R}{2\Rey},
\end{equation}
where $\rho_{f}$ and $\rho_{p}$ are the density of the fluid and
the particle, respectively. The constant, dimensionless vector of gravity is denoted by
$g$. The Stokes ($\St$) and Reynolds ($\Rey$) numbers are defined
as 
\begin{equation}
\St=\frac{2}{9}\left(\frac{a}{L}\right)^{2}\Rey,\ \ \ \Rey=\frac{UL}{\nu},
\end{equation}
where $a$ is the radius of the particle and $\nu$ denotes the kinematic viscosity of the fluid.

Equation \eqref{eq:MR_dimless0} is a system of nonlinear, fractional-order differential equations. The fractional order is due to the \emph{memory term}
\begin{equation}
\frac{\id}{\id t}\int_{t_0}^{t}\frac{w(s)}{\sqrt{t-s}}\id s=\int_{t_0}^{t}\frac{\dot{w}(s)}{\sqrt{t-s}}\id s+\frac{w(t_0)}{\sqrt{t-t_0}} \label{eq:RLid}
\end{equation}
where the identity is obtained by subsequent differentiation and integration-by-part \cite[see, e.g.,][]{FDE_podlubny}.
The memory term is a fractional derivative of order $1/2$ in the Riemann--Liouville sense \cite{FDE_podlubny,Daitche2013}. Physically, it represents the Basset--Boussinesq force \cite{boussinesq,basset1,basset2} resulting from the lagging boundary
layer development around the particle, as it moves through the fluid \citep{MR}. 

In the original derivation of the MR equation \cite{MR}, it is implicitly
assumed that the initial velocity of the particle $v_{0}$ is such
that $w(t_0)=0$ holds. Equation \eqref{eq:MR_dimless0}, however, is
the most general form of the MR equation which was derived later \cite{MR_initCond}
and allows for a general initial particle velocity $v_{0}$.

Without the memory term and for $w(t_0)=0$, the MR equation is an ordinary
differential equation, whose solutions are well known to be regular
for any smooth ambient velocity field $u(x,t)$. The memory term,
however, introduces complications in the analysis and numerical solution
of the equation.  It contains an implicit
term through the integral with an integrand depending on the particle acceleration $\ddot y$. 
Because of its implicit nature, it is not a priori
clear if the MR differential equation defines a dynamical system,
i.e., a process with a well-defined flow map.  

Furthermore, when nonzero, the unbounded term $w(t_0)/\sqrt{t-t_0}$ further complicates
equation \eqref{eq:MR_dimless0}, imparting an instantaneously infinite
force at the initial time. This term is often ignored for
convenience, even though its omission imposes a special constraint
on the initial particle velocity that is hard to justify physically \cite{MR_initCond}.

For the above reasons, the memory term has routinely been neglected
in studies of inertial particle dynamics (see, e.g., \citet{IP_maxey87,IP_babiano,IP_haller08}),
until recent studies demonstrated convincingly the quantitative and qualitative importance of the memory term (see, e.g., \citep{IP_candelier,toegel2006,gabrin2009} for experimental and \citep{IP_tel,guseva2013influence,Daitche_NJP} for numerical studies.).

In addition to theoretical difficulties, the memory term also complicates the numerical treatment of the full MR equation.
This equation is certainly not solvable with standard numerical schemes such as Runge--Kutta
algorithms. To this end, involved schemes have been developed for
numerical treatment of the memory term (see \citet{Daitche2013} and
references therein).

All these numerical schemes implicitly assume the existence and uniqueness
of solutions of the MR equation. The solutions can indeed be found
explicitly for certain simple velocity fields \cite{MR_exactSol,IP_candelier}. 
To the best of our knowledge, however, general
existence and uniqueness results have not been proven, and cannot
be directly concluded from existing results on broader classes of
evolution equations (see \cite{Burton2013,Burton2012,FDE_petras,FDE_podlubny} for
related but not applicable results on integro- and fractional-order differential equations).
In the absence of such results, the existence and regularity of solutions
for a nonlinear system of fractional-order differential equation,
such as the MR equation, is far from obvious. 

Here, we present the first proof of local existence and uniqueness
of mild solutions to the full MR equation. The solutions become classical
(strong) solutions to \eqref{eq:MR_dimless0} for initial conditions
satisfying $w(t_0)=0.$ Moreover, we show that both the mild and the strong solutions are continuously differentiable with respect to their initial conditions. As a consequence,  coherent-structure detection methods utilizing  the derivative of the flow map in the absence of the memory term  \cite{sapsis2009inertial} can also be employed in the present, more general context.

We start with re-writing the MR equation as a system of differential equations (see Eq. \eqref{eq:MR_compact} below) in terms of the particle position $y$ and the function $w$ defined in \eqref{eq:w}. Multi-dimensional reformulations of the MR equation have appeared before \cite{MR_exactSol,sapsis2011IP,Daitche2013} but remained inaccessible to general mathematical analysis due to the implicit dependence of their right-hand sides on $\dot y$.

Our formulation turns the MR equation into a nonlinear system of fractional-order differential equations in terms of $y$ and $w$. The standard techniques for the proof of existence and uniqueness of solutions of such equations assume Lipschitz continuity of the right hand side with respect to the $(y,w)$ variable \cite{FDE_petras,FDE_podlubny}. This assumption fails for the MR equation (see the term  $M_u(y,t)w$ in Eq. \eqref{eq:MR_compact} below). Therefore, as discussed in Section \S\ref{sec:EU}, modifications to the standard function spaces, estimates and assumptions are required.

\section{Preliminaries}

We start by letting the velocity of the inertial particle be $v:\mathbb{R}^{+}\rightarrow\mathbb{R}^{n}$
, and use this notation to rewrite \eqref{eq:MR_dimless0} as a
first-order system of equations
\begin{align}
\dot{y}= & \; v\nonumber\\
\dot{v}= & \, R\frac{Du}{Dt}+\left(1-\frac{3R}{2}\right)g+\frac{R}{2}\frac{D}{Dt}\left(u+\frac{\gamma}{10}\mu^{-1}\Delta u\right)\nonumber \\
\  & -\mu\left(v-u-\frac{\gamma}{6}\mu^{-1}\Delta u\right)\nonumber \\
 & -\kappa\mu^{1/2}\frac{\id}{\id t} \int_{t_0}^{t}\frac{w(s)}{\sqrt{t-s}}\id s ,\label{eq:MR_dimless}
\end{align}
with the function $w(t)$ defined as in \eqref{eq:w}. The memory term is written as a fractional-order derivative through identity \eqref{eq:RLid}. As earlier,
the material derivative $\frac{D}{Dt}\doteq\partial_{t}+u\cdot\nabla$
denotes a time derivative along a fluid trajectory. Also $\frac{d}{dt}\doteq\partial_{t}+v\cdot\nabla$
denotes temporal differentiation along the inertial trajectory $y(t)$.
The two derivatives are related by the identity 
\begin{equation}
\frac{d}{dt}=\frac{D}{Dt}+(v-u)\cdot\nabla.\label{eq:ident}
\end{equation}
For notational simplicity, we will also use the dot symbol for the
derivative $\frac{d}{dt}$.

%Using the identity \eqref{eq:ident} and the identity 
%\[
%\frac{\id}{\id t}\int_{t_0}^{t}\frac{w(s)}{\sqrt{t-s}}\id s=\int_{t_0}^{t}\frac{\dot{w}(s)}{\sqrt{t-s}}\id s+\frac{w(0)}{\sqrt{t}}
%\]
%for Riemann--Liouville fractional-order derivatives \cite{FDE_podlubny},
We rewrite (\ref{eq:MR_dimless}) in the more compact form 
\begin{align}
\dot{y}= & \; w+A_{u}(y,t),\nonumber\\
\dot{w}= & \,-\mu\, w-M_{u}(y,t)w-\kappa\mu^{1/2}\frac{\id}{\id t}\int_{t_0}^{t}\frac{w(s)}{\sqrt{t-s}}\id s+B_{u}(y,t),
\label{eq:MR_compact}
\end{align}
where 
\begin{align*}
A_{u}= & \; u+\frac{\gamma}{6}\mu^{-1}\Delta u,\\
B_{u}= & \left(\frac{3R}{2}-1\right)\left(\frac{Du}{Dt}-g\right)+\left(\frac{R}{20}-\frac{1}{6}\right)\gamma\mu^{-1}\frac{D}{Dt}\Delta u\\
\  & -\frac{\gamma}{6}\mu^{-1}\left[\nabla u+\frac{\gamma}{6}\mu^{-1}\nabla\Delta u\right]\Delta u,\\
M_{u}= & \nabla u+\frac{\gamma}{6}\mu^{-1}\nabla\Delta u,
\end{align*}
are known functions in terms of the fluid velocity $u$. The terms
$A_{u},\ B_{u}:\mathcal{D}\times R^{+}\rightarrow\mathbb{R}^{n}$
represent vector fields while $M_{u}:\mathcal{D}\times\mathbb{R}^{+}\rightarrow\mathbb{R}^{n\times n}$
is a tensor field. Note that equation (\ref{eq:MR_compact}) is linear
in $w$ and, for a typical fluid velocity field $u$, non-linear in
$y$. The corresponding initial conditions for \eqref{eq:MR_compact}
are $y(t_0)=y_{0}$ and $w(t_0)=w_{0}:=v_{0}-u(y_{0},t_0)-\frac{\gamma}{6}\mu^{-1}\Delta u(y_{0},t_0)$.

\section{Local existence and uniqueness}\label{sec:EU}

\subsection{Approach}

This section is devoted to proving the local existence and uniqueness
of solutions of \eqref{eq:MR_compact} under certain smoothness assumptions
on the fluid velocity field $u$.

Integrating equation \eqref{eq:MR_compact} formally, one obtains
\begin{align}
y(t)= & \; y_{0}+\int_{t_0}^{t}\big[w(s)+A_{u}(y(s),s)\big]\id s,\nonumber\\
w(t)= & \; w_{0}+\int_{t_0}^{t}\left[-\mu\, w(s)-M_{u}(y(s),s)w(s)-\kappa\mu^{1/2}\frac{w(s)}{\sqrt{t-s}}+B_{u}(y(s),s)\right]\id s,\label{eq:yw}
\end{align}
where, for notational simplicity, we have omitted the dependence of
$y$ and $w$ on $y_{0}$ and $w_{0}$. A \emph{mild solution} of
the MR equation is a function $(y(t),w(t))$ that satisfies the integral
equation \eqref{eq:yw}. The same solution is also a \emph{strong solution}
if it is smooth enough to also satisfy the differential form \eqref{eq:MR_compact}
of the MR equation.

Equation \eqref{eq:yw} can be viewed as a fixed point problem for
the map 
\begin{equation}
(P\Phi)(t)=\begin{pmatrix}y_{0}+\int_{t_0}^{t}\big[\eta(s)+A_{u}(\xi(s),s)\big]\id s\ \vspace{.22cm} \\ 
w_{0}+\int_{t_0}^{t}\left[-\left(\mu+\frac{\kappa\mu^{1/2}}{\sqrt{t-s}}+M_{u}(\xi(s),s)\right)\eta(s)+B_{u}(\xi(s),s)\right]\id s
\end{pmatrix},\label{eq:map_P}
\end{equation}
where $\Phi=(\xi,\eta)\in\mathbb{R}^{2n}$. We will establish the
existence of mild solutions to the MR equations by showing that $P$
has a unique fixed point on an appropriate function space under general
regularity assumptions on the fluid velocity $u$.

\subsection{Set-up}

We will use $|\cdot|$ to denote the Euclidean norm on $\mathbb{R}^{m}$ with $m\in\{n,2n\}$.
The induced operator norm of a square matrix acting on $\mathbb{R}^{m}$
is denoted by $\|\cdot\|$. For continuous functions defined on $\mathbb{R}^{m}$, 
we denote the supremum norm by $\|\cdot\|_{\infty}$.

Let $X_{T,K}$ denote the set of continuous functions mapping from
the interval $[t_0,t_0+T]$ into $\mathbb{R}^{m}$ that are uniformly bounded
by the constant $K>0$: 
\begin{equation}
X_{T,K}:=\{f\in C([t_0,t_0+T];\mathbb{R}^{m}):\;\|f\|_{\infty}\leq K\}.
\end{equation}
Since $(C([t_0,t_0+T];\mathbb{R}^{m}),\|\cdot\|_{\infty})$ is a Banach
space, the space $(X_{T,K},\|\cdot\|_{\infty})$ is a complete metric
space, for $X_{T,K}$ is a closed subset of $C([t_0,t_0+T];\mathbb{R}^{m})$. 

First, we would like to show that $P$ defined in \eqref{eq:map_P}
maps $X_{T,K}$ into itself. To this end, we need the following assumption.
\begin{description}
\item [{(H1)}] The velocity field $u(x,t)$ is three times continuously
differentiable in its arguments over the domain $\mathcal{D}\times\mathbb{R}^{+}$,
and its partial derivatives (including mixed partials) are uniformly
bounded and Lipschitz continuous up to order three.
\end{description}

\subsection{Existence and uniqueness of solutions}
Under assumption (H1), we obtain the following result:
\begin{lem}
Assume that (H1) holds. Then for any $y_{0}\in\mathcal{D}$ and $w_{0}\in\mathbb{R}^{n}$,
there exist $K>0$ large enough and $\delta>0$ small enough, such
that, for any $T\in[0,\delta]$, we have $P\colon X_{T,K}\rightarrow X_{T,K}$
.\label{lemma:selfMapping} \end{lem}
\begin{proof}
Under assumption (H1), the vector fields $A_{u},B_{u}:\mathcal{D}\times\mathbb{R}^{+}\rightarrow\mathbb{R}^{n}$
and the tensor field $M_{u}:\mathcal{D}\times\mathbb{R}^{+}\rightarrow\mathbb{R}^{n\times n}$
are continuous and uniformly bounded. Specifically, there exists a
constant $L_{b}>0$ such that 
\[
\|A_{u}\|_{\infty},\|B_{u}\|_{\infty},\|M_{u}\|_{\infty}\leq L_{b}.
\]
Then, based on eq. \eqref{eq:map_P}, the quantity $P\Phi$ satisfies
the estimate 
\begin{align*}
|P\Phi(t)|\leq & \|y_{0}+\int_{t_0}^{t}\big[\eta(s)+A_{u}(\xi(s),s)\big]\id s\|_{\infty}\\
 & +\|w_{0}+\int_{t_0}^{t}\left[\left(\mu+\frac{\kappa\mu^{1/2}}{\sqrt{t-s}}+M_{u}(\xi(s),s)\right)\eta(s)+B_{u}(\xi(s),s)\right]\id s\|_{\infty}\\
\leq & |y_{0}|+|w_{0}|+\|\eta\|_{\infty}\left((t-t_0)+\mu (t-t_0)+2\kappa\mu^{1/2}\sqrt{t-t_0}+L_{b}(t-t_0)\right)+2L_{b}(t-t_0)\\
\leq & |y_{0}|+|w_{0}|+\|\Phi\|_{\infty}\left((t-t_0)+\mu(t-t_0)+2\kappa\mu^{1/2}\sqrt{t-t_0}+L_{b}(t-t_0)\right)+2L_{b}(t-t_0).
\end{align*}
Now take $K=4\max\{|y_{0}|,|w_{0}|\}$ and $\delta>0$ small enough
such that 
\[
\delta+\mu\delta+2\kappa\mu^{1/2}\sqrt{\delta}+L_{b}\delta<\frac{1}{4},\ \ \ 2L_{b}\delta<\frac{K}{4}.
\]
Then, for any $T\in[0,\delta]$, $\|P\Phi\|_{\infty}\leq K$ given
that $\Phi\in X_{T,K}$. The continuity of $P\Phi:[t_0,t_0+T]\rightarrow\mathbb{R}^{2n}$
follows from assumption (H1) after one notes that, for $\eta\in X_{T,K}$,
the term $\int_{t_0}^{t}\frac{\eta(s)}{\sqrt{t-s}}\id s$ in \eqref{eq:map_P}
is continuous in $t$. 
\end{proof}
We establish the existence of a unique solution to \eqref{eq:yw}
by proving that $P$ is a contraction mapping on $X_{T,K}$ and hence
has a unique fixed point. 
\begin{lem}
\label{lemma: contraction}Assume that (H1) holds. Then there exists
$\delta>0$ such that, for any $T\in[0,\delta]$ and $\Phi_{1},\Phi_{2}\in X_{T,K}$,
\[
\|P\Phi_{1}-P\Phi_{2}\|_{\infty}\leq\frac{1}{2}\|\Phi_{1}-\Phi_{2}\|_{\infty}
\]
\end{lem}
\begin{proof}
Note that as a direct consequence of assumption (H1), the maps $A_{u}(\cdot,t),B_{u}(\cdot,t):\mathcal{D}\rightarrow\mathbb{R}^{n}$
and $M_{u}(\cdot,t):\mathcal{D}\rightarrow\mathbb{R}^{n\times n}$
are Lipschitz continuous, uniformly in time, i.e., there is a constant
$L_{c}>0$ such that, for any $t\in[t_0,t_0+T]$ and $y_{1},y_{2}\in\mathcal{D}$,
\[
|A_{u}(y_{1},t)-A_{u}(y_{2},t)|\leq L_{c}|y_{1}-y_{2}|,
\]
\[
|B_{u}(y_{1},t)-B_{u}(y_{2},t)|\leq L_{c}|y_{1}-y_{2}|,
\]
\begin{equation}
\|M_{u}(y_{1},t)-M_{u}(y_{2},t)\|\leq L_{c}|y_{1}-y_{2}|.\label{eq:H2}
\end{equation}

Let $\Phi_{1},\Phi_{2}\in X_{T,K}$, where $\Phi_{i}=(\xi_{i},\eta_{i})$.
Using the above inequalities, we have 
\begin{align*}
|(P\Phi_{1})(t)-(P\Phi_{2})(t)|\leq & \int_{t_0}^{t}\big(|\eta_{1}(s)-\eta_{2}(s)|+|A_{u}(\xi_{1}(s),s)-A_{u}(\xi_{2}(s),s)|\big)\id s+\\
 & \int_{t_0}^{t}\Big[\mu|\eta_{1}(s)-\eta_{2}(s)|+\kappa\mu^{1/2}\frac{|\eta_{1}(s)-\eta_{2}(s)|}{\sqrt{t-s}}+\\
 & \ \ \ \ \ \ \ \ |M_{u}(\xi_{1}(s),s)\eta_{1}(s)-M_{u}(\xi_{2}(s),s)\eta_{2}(s)|+\\
 & \ \ \ \ \ \ \ \ |B_{u}(\xi_{1}(s),s)-B_{u}(\xi_{2}(s),s)|\Big]\id s\\
\leq & \left[(t-t_0)+\mu (t-t_0)+2\kappa\mu^{1/2}\sqrt{t-t_0}\right]\|\eta_{1}-\eta_{2}\|_{\infty}+\\
 & 2L_{c}(t-t_0)\|\xi_{1}-\xi_{2}\|_{\infty}+\\
 & \int_{t_0}^{t}\big[|M_{u}(\xi_{1}(s),s)\eta_{1}(s)-M_{u}(\xi_{1}(s),s)\eta_{2}(s)|+\\
 & \ \ \ \ \ \ \ |M_{u}(\xi_{1}(s),s)\eta_{2}(s)-M_{u}(\xi_{2}(s),s)\eta_{2}(s)|\big]\id s\\
\leq & \left[(t-t_0)+\mu (t-t_0)+2\kappa\mu^{1/2}\sqrt{t-t_0}+L_{b}(t-t_0)\right]\|\eta_{1}-\eta_{2}\|_{\infty}+\\
     & (2+K)L_{c}(t-t_0)\|\xi_{1}-\xi_{2}\|_{\infty},
\end{align*}
where we used the fact that 
\begin{align*}
\int_{t_0}^{t}\Big|M_{u}(\xi_{1}(s),s)\eta_{1}(s)-M_{u}(\xi_{2}(s),s)\eta_{2}(s)\Big|\id s= & \ \\
\int_{t_0}^{t}\Big|M_{u}(\xi_{1}(s),s)(\eta_{1}(s)-\eta_{2}(s))+ & \left(M_{u}(\xi_{1}(s),s)-M_{u}(\xi_{2}(s),s)\right)\eta_{2}(s)\Big|\id s\\
\leq\int_{t_0}^{t}\|M_{u}(\xi_{1}(s),s)\||\eta_{1}(s)-\eta_{2}(s)|\id s+ & \int_{t_0}^{t}\|M_{u}(\xi_{1}(s),s)-M_{u}(\xi_{2}(s),s)\||\eta_{2}(s)|\id s\\
\leq (t-t_0)L_{b}\|\eta_{1}-\eta_{2}\|_{\infty}+ & (t-t_0)L_{c}\|\eta_{2}\|_{\infty}\|\xi_{1}-\xi_{2}\|_{\infty}.
\end{align*}
Therefore, one can take $\delta>0$ small enough such that, for any
$t\in[t_0,t_0+\delta]$, 
\begin{align*}
|(P\Phi_{1})(t)-(P\Phi_{2})(t)|\leq & \frac{1}{4}\left(\|\xi_{1}-\xi_{2}\|_{\infty}+\|\eta_{1}-\eta_{2}\|_{\infty}\right)\\
\leq & \frac{1}{2}\|\Phi_{1}-\Phi_{2}\|_{\infty}.
\end{align*}
Here, for the last inequality, we have used the fact that $\|\xi\|_{\infty}+\|\eta\|_{\infty}<2\|\Phi\|_{\infty}$.
Hence, we obtain the contraction property 
\[
\|P\Phi_{1}-P\Phi_{2}\|_{\infty}\leq\frac{1}{2}\|\Phi_{1}-\Phi_{2}\|_{\infty},
\]
 as claimed.
\end{proof}
Lemma \ref{lemma: contraction} leads to our main existence result.
\begin{thm}
\label{thm:Local-existence}{[}Local existence of mild solutions{]}
Assume that (H1) holds. Then for any initial condition $\left(y_{0},w_{0}\right)\in\mathcal{D\times\mathbb{R}}^{n}$,
there exists $\delta>0$ such that over the time interval $[t_0,t_0+\delta]$,
the integral equation \eqref{eq:yw} has a unique solution $(y(t),w(t))$
with $\left(y(t_0),w(t_0)\right)=\left(y_{0},w_{0}\right)$. As consequence,
the function $y(t)$ is a mild solution of the original form \eqref{eq:MR_dimless0}
of the Maxey--Riley equation.\end{thm}
\begin{proof}
By Lemma \ref{lemma: contraction}, for any $y_{0}\in\mathcal{D}$
and $w_{0}\in\mathbb{R}^{n}$, there exist $K>0$ and $\delta>0$
such that for any $T\in[0,\delta]$, the map $P:X_{T,K}\rightarrow X_{T,K}$
is a contraction on the complete metric space $X_{T,K}$. As a consequence,
the mapping $P$ has a unique fixed point $\left(y,w\right):[t_0,t_0+\delta]\rightarrow\mathcal{D\mathcal{\times\mathbb{R}}}^{n}$.
By the definition of $P$, this fixed point solves the integral equation
\eqref{eq:yw}, and hence is a mild solution of \eqref{eq:MR_compact},
or equivalently, of \eqref{eq:MR_dimless0}.\end{proof}
\begin{rem}
The solution $y(t)$ is, in general, not a strong solution of \eqref{eq:MR_dimless0},
because it is only once continuously differentiable at $t=t_0$, and
hence only satisfies the integrated form of $w(t)$. The following example demonstrates the lack of existence of strong solutions in a simple case where the $w$-equation in \eqref{eq:MR_compact} can be solved explicitly.
\end{rem}

\begin{example}
For a uniform fluid at rest (i.e., $u\equiv 0$), if we neglect the effect of gravity (i.e., set $g=0$), $A_u$, $B_u$ and $M_u$ in equation \eqref{eq:MR_compact} vanish. Then, the equation for $w$ reduces to 
$$\dot w+\kappa\mu^{1/2}\frac{\id}{\id t}\int_{0}^{t}\frac{w(s)}{\sqrt{t-s}}\id s+\mu w=0,$$
with the initial time $t_0=0$ and an arbitrary initial condition $w(0)=w_0$. Taking the Laplace transform of this equation, we obtain
$$W(p)=\frac{1}{p+Gp^{1/2}+\mu}w_0,$$
where $G=\sqrt{9R\mu/2}$ and $W$ denotes the Laplace transform of $w$. For $R< 8/9$, the inverse Laplace transform yields the exact solution
$$w(t)=w_0\Big\{\e^{-\alpha t}\cos(\beta t)+\frac{G^2}{2\beta}\e^{-\alpha t}\sin(\beta t)-\frac{G}{\sqrt{\pi}}\int_{0}^{t}\frac{\e^{-\alpha s}\cos(\beta s)-(\alpha/\beta)\e^{-\alpha s}\sin(\beta s)}{\sqrt{t-s}}\id s\Big\},$$
with $\alpha=\mu(1-9R/4)$ and $\beta=G\sqrt{\mu(1-9R/8)}$. Defining
$$c(s)=\e^{-\alpha s}\cos(\beta s)-(\alpha/\beta)\e^{-\alpha s}\sin(\beta s),$$
and taking the derivative of $w$ with respect to time $t$, we obtain
$$\dot w(t)=w_0\Big\{\left(\frac{G^2}{2}-\alpha\right)\e^{-\alpha t}\cos(\beta t)-\left(\frac{\alpha G^2}{2\beta}+\beta\right)\e^{-\alpha t}\sin(\beta t)-\frac{G}{\sqrt{\pi}}\int_{0}^{t}\frac{\dot c(s)}{\sqrt{t-s}}\id s -\frac{G}{\sqrt{\pi t}} \Big\}.$$

For any $T>0$, the first three terms in $\dot w$ are continuous over the time interval $[0,T]$. The last term $\frac{G}{\sqrt{\pi t}}$, however, is discontinuous at $t=0$. This concludes our example showing that, in general, the MR equation with non-zero initial condition $w_0$ only admits mild solutions.$\ \ \blacksquare$
\end{example}

As mentioned in the Introduction, the original form of the MR equation
\cite{MR} assumes the initial velocity $w(t_0)=0$. This assumption
is mathematically convenient, as it removes the unbounded term from
\eqref{eq:MR_dimless0}. Physically, however, the assumption is artificial,
and cannot be enforced at the release of an inertial particle. 

Nevertheless, $w(t_0)=0$ has been routinely assumed in various studies
of the MR equation (see, e.g., \citet{IP_babiano,IP_candelier,IP_tel})
as an important special case. We now show that under this special
assumption, the MR equation in fact has strong solutions.
\begin{thm}
Assume that (H1) holds. Then for any $y_{0}\in\mathcal{D}$, there
exists $\delta>0$ such that, over the time interval $[t_0,t_0+\delta]$,
the Maxey-Riley equation \eqref{eq:MR_dimless0} has a unique solution
satisfying $y(t_0)=y_{0}$ and $w(t_0)=0$. \label{thm:ExUnStrong} \end{thm}
\begin{proof}
See Appendix \ref{app:proofStrongSol}.
\end{proof}

\subsection{Regularity of solutions}

Here we show the differentiability of the solutions of \eqref{eq:MR_compact}
with respect to the initial condition $(y_{0},w_{0})$. Assume that
a solution $(y(t),w(t))$ is differentiable at $(y_{0},w_{0})$ and
denote the derivative of $y$ and $w$ with respect to $(y_{0},w_{0})$
by $Dy$ and $Dw$, respectively. 

Differentiating \eqref{eq:MR_compact} formally and integrating in
time, we obtain that $Dy,Dw:\mathbb{R}^{+}\rightarrow\mathbb{R}^{n\times2n}$
must satisfy 
\begin{align}
Dy(t)= & \big(I_{n}|O_{n}\big)+\int_{t_0}^{t}\left[Dw(s)+\nabla A_{u}(y(s),s)Dy(s)\right]\id s,\nonumber \\
Dw(t)= & \big(O_{n}|I_{n}\big)+\int_{t_0}^{t}\Big[-\mu Dw(s)-\mathcal{L}(y(s),w(s),s)Dy(s)-M_{u}(y(s),s)Dw(s)\nonumber \\
 & \ \ \ \ \ \ \ \ \ \ \ \ \ \ \ \ \ \ \ \ \ \ -\kappa\mu^{1/2}\frac{Dw(s)}{\sqrt{t-s}}+\nabla B_{u}(y(s),s)Dy(s)\Big]\id s,\label{eq:mapDPhi}
\end{align}
where $\mathcal{L}$ denotes the $n\times n$ matrix given by 
\[
\mathcal{L}_{ij}(y(s),w(s),s)=\sum_{k}\frac{\partial M_{ik}}{\partial y_{j}}\Big|_{(y(s),s)}w_{k}(s).
\]
The matrices $I_{n}$ and $O_{n}$ denote the identity and null matrices
on $\mathbb{R}^{n\times n}$.

The differentiability of the solution $(y,w)$ with respect to the
initial condition $(y_{0},w_{0})$, therefore, is equivalent to the
existence and uniqueness of solutions to the equations \eqref{eq:mapDPhi}.
We show that under further regularity assumptions on the fluid velocity
$u$, a unique solution to these equations does exist. In particular,
we need the following assumption: 
\begin{description}
\item [{(H2)}] The velocity field $u(x,t)$ is four times continuously
differentiable in its arguments over the domain $\mathcal{D}\times\mathbb{R}^{+}$.
Its partial derivatives (including mixed partials) are uniformly bounded
and Lipschitz continuous up to order three. \label{ass:C1regularity} \end{description}
\begin{thm}
Assume that (H2) holds. Then for any $y_{0}\in\mathcal{D}$ and $w_{0}\in\mathbb{R}^{n}$,
there exists $\delta>0$ small enough such that, a unique mild solution
$(y(t),w(t))$ of \eqref{eq:MR_compact} exists over the time interval
$[t_0,t_0+\delta]$, and is continuously differentiable with respect to
its initial condition $(y_{0},w_{0})$. \end{thm}
\begin{proof}
Note that the map $\mathcal{P}$ defined by the right hand side of
\eqref{eq:mapDPhi} is linear in $Dy$ and $Dw$. It follows from
assumption (H2) that $\mathcal{P}$ maps $C([t_0,t_0+\delta];\mathbb{R}^{2n\times2n})$
into itself for any $\delta\in\mathbb{R}^{+}$. Furthermore, for $\delta>0$
small enough, the map $\mathcal{P}$ is a contraction $C([t_0,t_0+\delta];\mathbb{R}^{2n\times2n})$
by an argument similar to Lemma \eqref{lemma: contraction} (omitted
here for brevity). Therefore, there are unique derivatives $Dy,Dw:[t_0,t_0+\delta]\rightarrow\mathbb{R}^{n\times2n}$
that belong to the function space $C([t_0,t_0+\delta];\mathbb{R}^{n\times2n})$
and solve equations \eqref{eq:mapDPhi}. \end{proof}
\begin{rem}
For the special case $w(t_0)=0$, one can similarly show that the strong
solution $(y(t),w(t))$ is differentiable with respect to the initial
position $y_{0}$. 
\end{rem}

\section{Conclusion}

We have proved the local existence and uniqueness of solutions of
the Maxey--Riley (MR) equation. In the most general case, the solutions
exist only in a weak sense. This is consistent with the physics of
the problem, because an initial velocity mismatch between the ambient
fluid and the particle creates a vorticity layer around the particle
with high drag. This drag force is modeled in the MR equation by a
term proportional to $1/\sqrt{t-t_0}$, which is singular but integrable.
As a result, the solution of the MR equation is continuous but only
differentiable for $t>t_0$.

In theoretical and numerical investigations of the MR equation, it
is routinely assumed that the relative velocity term $w(t)$ is chosen
in a way that eliminates the infinitely large force at time $t=t_0$.
We have shown that under this assumption, a unique strong solution
exists to the MR equation. Moreover, both the mild and the strong
solutions are differentiable with respect to their initial conditions.

Remaining challenges for the MR equations include global existence
and uniqueness and an asymptotic analysis of the solutions, at least
for small inertial particles.

\section*{Acknowledgment}

We are grateful to Tibor Krisztin for his insights on the Maxey--Riley
equations, and for pointing out Ref. \cite{Burton2013}. We would also like to acknowledge  useful conversations with Martin Maxey and Tam\'as T\'el on  the Maxey-Riley equations. We are also thankful to Daniel Karrasch for fruitful discussions on an early draft of this manuscript. 

\begin{appendices}
\section{Proof of Theorem \ref{thm:ExUnStrong}}\label{app:proofStrongSol} 
First, we slightly reformulate the MR
equation. If continuously differentiable solutions to equation \eqref{eq:MR_compact}
exist, then the integral term in the equation can be re-written as
\[
\frac{\id}{\id t}\int_{t_0}^{t}\frac{w(s)}{\sqrt{t-s}}\id s=\int_{t_0}^{t}\frac{\dot{w}(s)}{\sqrt{t-s}}\id s,
\]
since $w(t_0)=0$. As a consequence, the MR equation \eqref{eq:MR_compact}
can be written as 
\begin{align}
\dot{y}= & \; w+A_{u}(y,t),\nonumber \\
\dot{w}= & \,-\mu\, w-M_{u}(y,t)w-\kappa\mu^{1/2}\int_{t_0}^{t}\frac{\dot{w}(s)}{\sqrt{t-s}}\id s+B_{u}(y,t).\label{eq:MR_compact2}
\end{align}
 Now, we would like to show that this latter equation, in fact, admits
continuously differentiable solutions satisfying $y(t_0)=y_{0}$ and
and $w(t_0)=0$. Our proof will differ from the proof of Theorem \ref{thm:Local-existence}.
The main ideas follow those of \citet{Burton2012}, although the details
are quite different. In particular, the results of \cite{Burton2012}
do not apply in our context.

We need to show that there are unique bounded continuous functions
$\phi,\psi:\mathcal{[}t_0,t_0+T]\rightarrow\mathbb{R}^{n}$ such that the
functions 
\[
y(t)=y_{0}+\int_{t_0}^{t}\phi(s)\id s,
\]
\begin{equation}
w(t)=\int_{t_0}^{t}\psi(s)\id s,\label{eq:formal_Sol}
\end{equation}
solve equation (\ref{eq:MR_compact2}). For notational simplicity,
we omit the dependence of $y$, $w$, $\phi$ and $\psi$ on the initial
condition $y_{0}$.

Substituting $y(t)$ and $w(t)$ in (\ref{eq:MR_compact2}), we obtain
\begin{align}
\phi(t)= & \;\int_{t_0}^{t}\psi(s)\id s+A_{u}\left(y_{0}+\int_{t_0}^{t}\phi(s)\id s,t\right),\nonumber\\
\psi(t)= & \,-\mu\,\int_{t_0}^{t}\psi(s)\id s-M_{u}\left(y_{0}+\int_{t_0}^{t}\phi(s)\id s,t\right)\int_{t_0}^{t}\psi(s)\id s\nonumber \\
 & -\kappa\mu^{1/2}\int_{t_0}^{t}\frac{\psi(s)}{\sqrt{t-s}}\id s+B_{u}\left(y_{0}+\int_{t_0}^{t}\phi(s)\id s,t\right).
\end{align}

The right-hand sides of these equations define a mapping $P$ as 
\begin{equation}
(P\Phi)(t)=\begin{pmatrix}\int_{t_0}^{t}\psi(s)\id s+A_{u}(y(t),t)\\
\ \\
-\int_{t_0}^{t}\left[\mu+\frac{\kappa\mu^{1/2}}{\sqrt{t-s}}+M_{u}(y(t),t)\right]\psi(s)\id s+B_{u}(y(t),t)
\end{pmatrix},\label{eq:map_P2}
\end{equation}
where $\Phi=(\phi,\psi)\in\mathbb{R}^{2n}$ and $y(t)=y_{0}+\int_{t_0}^{t}\phi(s)\id s$.

We will show that the mapping $P$ has a unique fixed point in $X_{T,K}$
for some $T,K>0$. Then the existence of the above mentioned solution
of (\ref{eq:MR_compact2}) follows directly.

The following lemma shows that for an appropriate choice of $T$ and
$K$, $P$ maps $X_{T,K}$ into itself.
\begin{lem}
Assume that (H1) holds. Then for $K\geq 4L_{b}$ and any $y_{0}\in\mathcal{D}$,
there exists $\delta>0$ such that, for any $T\in[0,\delta]$, we have
$P:X_{T,K}\rightarrow X_{T,K}$. \label{lemma:selfMapping2} \end{lem}
\begin{proof}
The continuity of $P\Phi:\mathbb{R}^{+}\rightarrow\mathbb{R}^{2n}$
follows from assumption (H1). We also have 
\begin{align*}
|P\Phi(t)|\leq & \|\int_{t_0}^{t}\psi(s)\id s+A_{u}(y(t),t)\|_{\infty}\\
 & +\|-\int_{t_0}^{t}\left[\mu+\frac{\kappa\mu^{1/2}}{\sqrt{t-s}}+M_{u}(y(t),t)\right]\psi(s)\id s+B_{u}(y(t),t)\|_{\infty}\\
\leq & \|\psi\|_{\infty}\left((t-t_0)+\mu (t-t_0)+2\kappa\mu^{1/2}\sqrt{t-t_0}+L_{b}(t-t_0)\right)+2L_{b}\\
\leq & \|\Phi\|_{\infty}\left((t-t_0)+\mu (t-t_0)+2\kappa\mu^{1/2}\sqrt{t-t_0}+L_{b}(t-t_0)\right)+2L_{b}\\
\leq & K\left((t-t_0)+\mu (t-t_0)+2\kappa\mu^{1/2}\sqrt{t-t_0}+L_{b}(t-t_0)\right)+\frac{K}{2}\\
\end{align*}
If $\delta>0$ is small enough such that $\delta+\mu \delta+2\kappa\mu^{1/2}\sqrt{\delta}+L_{b}\delta\leq 1/2$, we have $\|P\Phi\|_{\infty}\leq K$; and
hence $P\Phi\in X_{T,K}$ for any $T\in[0,\delta]$. 
\end{proof}
We now fix the constant $K=4L_{b}$ in the following. We show that
the map $P$ is a contraction mapping on the space $X_{T,K}$.
\begin{lem}
There is $\delta>0$ such that, for any $T\in[0,\delta]$ and $\Phi_{1},\Phi_{2}\in X_{T,K}$,
\[
\|P\Phi_{1}-P\Phi_{2}\|_{\infty}\leq\frac{1}{2}\|\Phi_{1}-\Phi_{2}\|_{\infty}
\]
\label{lem:contraction2} 
\end{lem}
\begin{proof}
Let $\Phi_{1},\Phi_{2}\in X_{T,K}$ where $\Phi_{i}=(\phi_{i},\psi_{i})^{\top}$.
We have 
\allowdisplaybreaks{
\begin{align*}
|(P\Phi_{1})(t)-(P\Phi_{2})(t)|\leq & \int_{t_0}^{t}|\psi_{1}(s)-\psi_{2}(s)|\id s+|A_{u}(y_{1}(t),t)-A_{u}(y_{2}(t),t)|\\
 & +\mu\int_{t_0}^{t}|\psi_{1}(s)-\psi_{2}(s)|\id s\\
 & +\kappa\mu^{1/2}\int_{t_0}^{t}\frac{|\psi_{1}(s)-\psi_{2}(s)|}{\sqrt{t-s}}\id s\\
 & +|M_{u}(y_{1}(t),t)\int_{t_0}^{t}\psi_{1}(s)\id s-M_{u}(y_{2}(t),t)\int_{t_0}^{t}\psi_{2}(s)\id s|\\
 & +|B_{u}(y_{1}(t),t)-B_{u}(y_{2}(t),t)|\\
\leq & \int_{t_0}^{t}|\psi_{1}(s)-\psi_{2}(s)|\id s+L_{c}\int_{t_0}^{t}|\phi_{1}(s)-\phi_{2}(s)|\id s\\
 & +\mu\int_{t_0}^{t}|\psi_{1}(s)-\psi_{2}(s)|\id s\\
 & +\kappa\mu^{1/2}\int_{t_0}^{t}\frac{|\psi_{1}(s)-\psi_{2}(s)|}{\sqrt{t-s}}\id s\\
 & +L_{b}\int_{t_0}^{t}|\psi_{1}(s)-\psi_{2}(s)|\id s\\
 & +L_{c}\big((t-t_0)\|\psi_{2}\|_{\infty}\big)\int_{t_0}^{t}|\phi_{1}(s)-\phi_{2}(s)|\id s\\
 & +L_{c}\int_{t_0}^{t}|\phi_{1}(s)-\phi_{2}(s)|\id s,\\
\end{align*}
}
where we have used the Lipschitz continuity of $A_{u}(\cdot,t)$,
$B_{u}(\cdot,t)$ and $M_{u}(\cdot,t)$. We also used the fact that
\begin{align*}
\Big|M_{u}(y_{1},t)\int_{t_0}^{t}\psi_{1}(s)\id s-M_{u}(y_{2},t)\int_{t_0}^{t}\psi_{2}(s)\id s\Big|= & \ \\
\Big|M_{u}(y_{1},t)\int_{t_0}^{t}(\psi_{1}(s)-\psi_{2}(s))\id s+ & \left(M_{u}(y_{1},t)-M_{u}(y_{2},t)\right)\int_{t_0}^{t}\psi_{2}(s)\id s\Big|\\
\leq\|M_{u}(y_{1},t)\|\int_{t_0}^{t}|\psi_{1}(s)-\psi_{2}(s)|\id s+ & \|M_{u}(y_{1},t)-M_{u}(y_{2},t)\|\int_{t_0}^{t}|\psi_{2}(s)|\id s.
\end{align*}
As a result, we obtain 
\begin{align*}
|(P\Phi_{1})(t)-(P\Phi_{2})(t)|\leq & \Big((t-t_0)+\mu (t-t_0)+2\kappa\mu^{1/2}\sqrt{t-t_0}+L_{b}(t-t_0)\Big)\|\psi_{1}-\psi_{2}\|_{\infty}\\
 & +\Big(2L_{c}(t-t_0)+L_{c}K(t-t_0)^{2}\Big)\|\phi_{1}-\phi_{2}\|_{\infty}.
\end{align*}
Therefore, one can take $\delta>0$ small enough such that, for any
$t\in[t_0,t_0+\delta]$, 
\begin{align*}
|(P\Phi_{1})(t)-(P\Phi_{2})(t)|\leq & \frac{1}{4}\left(\|\phi_{1}-\phi_{2}\|_{\infty}+\|\psi_{1}-\psi_{2}\|_{\infty}\right)\\
\leq & \frac{1}{2}\|\Phi_{1}-\Phi_{2}\|_{\infty}.
\end{align*}
Hence we get the contraction property 
\[
\|P\Phi_{1}-P\Phi_{2}\|_{\infty}\leq\frac{1}{2}\|\Phi_{1}-\Phi_{2}\|_{\infty}.
\]
\end{proof}

Since $P$ is a contraction mapping on the complete metric space $X_{T,K}$,
it has a unique fixed point in $X_{T,K}$. Therefore, there are unique
continuous functions $\phi,\psi:[t_0,t_0+\delta]\rightarrow\mathbb{R}^{n}$
such that the functions $y,w$ defined by \eqref{eq:formal_Sol} solve
the MR equation \eqref{eq:MR_compact2} and satisfy $y(t_0)=y_{0}$
and $w(t_0)=0$. This concludes the proof of Theorem \ref{thm:ExUnStrong}. 

\end{appendices}

%\bibliographystyle{unsrtnat}
%\bibliography{../bibliog.bib}

\end{document}